\documentclass[secnum,leqno]{article}
\usepackage{amssymb}
\usepackage{amsmath}
\usepackage{amsthm}
\usepackage[dvipdfm]{graphicx} 
\usepackage{verbatim,enumerate}
\usepackage[active]{srcltx}
\usepackage{array}
\newcolumntype{L}{>{\displaystyle}l}
\newcolumntype{C}{>{\displaystyle}c}
\newcolumntype{R}{>{\displaystyle}r}
\usepackage[usenames]{color}
 \newtheorem{theorem}{Theorem}[section]
 \newtheorem*{theorem*}{Theorem}
 \newtheorem*{lemma*}{Lemma}
 \newtheorem{proposition}[theorem]{Proposition}
 
 \newtheorem{fact*}{Fact}
 \newtheorem{lemma}[theorem]{Lemma}
 \newtheorem{corollary}[theorem]{Corollary}
\theoremstyle{definition}
 \newtheorem{definition}[theorem]{Definition}
 
 \newtheorem*{remark*}{Remark}

\numberwithin{equation}{section}

\newcommand{\R}{\boldsymbol{R}}

\newcommand{\rank}{\operatorname{rank}}

\renewcommand{\phi}{\varphi}
\newcommand{\sign}{\operatorname{sgn}}

\newcommand{\Hom}{\operatorname{Hom}}
\newcommand{\coker}{\operatorname{coker}}

\newcommand{\A}{\mathcal{A}}

\newcommand{\ep}{\varepsilon}

\newcommand{\pmt}[1]{{\begin{pmatrix} #1  \end{pmatrix}}}

\newcommand{\trans}[1]{\vphantom{#1}^t\!#1}
\newcommand{\diff}[2]{\operatorname{Diff}^{#2}(#1)}
\newcommand{\mycomment}[1]{}
\newcommand{\red}[1]{\textcolor{Black}{#1}}
\begin{document}
\begin{center}
{\large {\bf Criteria for Morin singularities into higher dimensions}}
\\[2mm]
{\today}
\\[2mm]
Kentaro Saji\\
\end{center}
\renewcommand{\thefootnote}{\fnsymbol{footnote}}
\footnote[0]{
Dedicated to Professor Yasutaka Nakanishi
on the occasion of his\/ $60$th birthday}
\footnote[0]{
Partly supported by 
Japan Society for the Promotion of Science (JSPS)
and 
Coordenadoria de Aperfei\c{c}oamento de Pessoal de N\'ivel Superior
under the Japan-Brazil research cooperative
program and
Grant-in-Aid for Scientific Research (C) No. 26400087,
from JSPS.}

\begin{abstract}      
We give criteria for Morin singularities
into higher dimensions.
As an application, we study the number of
$\A$-isotopy classes of Morin singularities.
\end{abstract}

\section{Introduction}
A map-germ
$f:(\R^m,0)\to(\R^n,0)$ $(m<n)$ is called an
{\em $r$-Morin singularity\/}
($m\geq r(m-n+1)$)
if it is $\A$-equivalent to
the following map-germ at the origin:
\begin{equation}\label{eq:morinnor1}
h_{0,r}:x
\mapsto
\big(x_1,\ldots,x_{m-1},h_1(x),\ldots,
h_{n-m+1}(x)\big),
\end{equation}
where $x=(x_1,\ldots,x_m)$, and
\begin{equation}\label{eq:morinnor2}
\begin{array}{rcL}
h_i(x)&=&
\sum_{j=1}^rx_{(i-1)r+j}x_m^j\quad(i=1,\ldots,n-m),\\
h_{n-m+1}(x)&=&
\sum_{j=1}^{r-1}x_{(n-m)r+j}x_m^j+x_m^{r+1}.
\end{array}
\end{equation}
We say that two map-germs 
$f,g:(\R^m,0)\to(\R^n,0)$ are {\em $\A$-equivalent\/}
if there exist 
diffeomorphism-germs\hfill 
$\phi:(\R^m,0)\to(\R^m,0)$ \hfill and \hfill 
$\Phi:(\R^n,0)\to(\R^n,0)$ \hfill  such \hfill that\\
$\Phi\circ f\circ \phi=g$ holds.
Morin singularities are stable, and conversely, corank one and stable
germs are Morin singularities. 
This means that Morin singularities are fundamental and
frequently appear as singularities of
maps from one manifold to another.
Morin gave a characterization of them by a
transversality of the Thom-Boardman singularity set
and also gave criteria for germs of a
normalized 
form
$\big(x_1,\ldots,x_{m-1},g_1(x),\ldots,
g_{n-m+1}(x)\big)$.
Morin singularities are also characterized using
the intrinsic derivative due to 
Porteous (\cite{porteous} see also \cite{ando,GG}).
Criteria for singularities without using
normalization are not only more convenient 
but also indispensable in some cases.
We call
criteria without normalizing {\em general criteria}\/.
In fact, in the case of wave front
surfaces in 3-space, general criteria for 
cuspidal edges and swallowtails were given
in \cite{krsuy},
where we studied the local and global behavior of flat fronts 
in hyperbolic 3-space using them.
Recently general criteria for other singularities and 
several applications of them have been given
\red{(see \cite{ishimachi,ist,kruy,nishi,jgea,suyak,syy})}.
In this paper, we give general criteria for Morin singularities.
Using them, we give applications to singularities of
ruling maps
and $\A$-isotopy of Morin singularities.
\red{See \cite{df,fukuda,saeki,saekisaku,sst,szu} for other 
investigations of Morin singularities.}

\section{Singular set and restriction of a map to the singular set}
Let 
$f:(\R^m,0)\to(\R^n,0)$ $(m<n)$ be a map-germ.
We assume $\rank df_0=m-1$.
Then one can take a coordinate system
satisfying
\begin{equation}\label{eq:coord}
\rank d(f_1,\ldots,f_{m-1})=m-1,\quad
\text{and}\quad
(df_{m-1+i})_0=0\quad(i=1,\ldots,n-m+1),
\end{equation}
where $f=(f_1,\ldots,f_n)$.
We set
\begin{equation}\label{eq:lambda}
\begin{array}{rcl}
\lambda_i
&=&
\det(df_1,\ldots,df_{m-1},df_{m-1+i})\quad
(i=1,\ldots,n-m+1),\quad\text{and}\\
\Lambda&=&(\lambda_1,\ldots,\lambda_{n-m+1})
:(\R^m,0)\to(\R^{n-m+1},0).
\end{array}
\end{equation}
Let $0$ be a singular point of $f$.
\red{The} singular point $0$ of $f$ is said to be {\em non-degenerate\/}
if $\rank d\Lambda_0=n-m+1$ holds.
This definition does not depend on the choice of
coordinate system on the source, nor on the target:
\begin{lemma}\label{lem:nondeg}
Non-degeneracy does not depend on the choice of
coordinate system on the source, nor on the target
satisfying\/ \eqref{eq:coord}.
Furthermore, if\/ $0$ is a non-degenerate singular point of\/ $f$,
then the set of singular points\/ $S(f)$ is a manifold.
\end{lemma}
\begin{proof}
Let $\phi:(\R^m,0)\to(\R^m,0)$ be
a diffeomorphism-germ.
Then 
$$
\det\big(d(f_1\circ\phi),\ldots,d(f_{m-1}\circ\phi),
d(f_{m-1+i}\circ\phi)\big)
=
(\lambda_i\red{\circ\phi})\,\det d\phi\ 
(i=1,\ldots,n-m+1)
$$
holds. Thus non-degeneracy does not depend on the choice of
coordinate systems on the source satisfying \eqref{eq:coord}.
Next, let us assume that
$\rank d(f_1,\ldots,f_{m-1})=m-1$ and 
$(df_{m-1+i})_0=0$ $(i=1,\ldots,n-m+1)$.
Since non-degeneracy does not depend on the choice of
coordinate systems on the source,
we may assume $f$ is written as
\begin{equation}\label{eq:normal}
f(x)=\big(x_1,\ldots,x_{m-1},f_m(x),\ldots,f_{n}(x)\big),\quad
(df_{m-1+i})_0=0,\quad (i=1,\ldots,n-m+1),
\end{equation}
where $x=(x_1,\ldots,x_m)$.
Let us take a diffeomorphism-germ
$\Phi=(\Phi_1,\ldots,\Phi_n):(\R^n,0)\to(\R^n,0)$.
By assumption, we may assume that 
$$
d\Phi_0
=
\left(
\begin{array}{ccc|ccc}
(\Phi_1)_{X_1}&\cdots&(\Phi_1)_{X_{m-1}}&
(\Phi_1)_{X_m}&\cdots&(\Phi_1)_{X_{n}}\\
\vdots&\vdots&\vdots&\vdots&\vdots&\vdots\\
(\Phi_{m-1})_{X_1}&\cdots&(\Phi_{m-1})_{X_{m-1}}&
(\Phi_{m-1})_{X_m}&\cdots&(\Phi_{m-1})_{X_{n}}\\
\hline
&&&
(\Phi_{m})_{X_m}&\cdots&(\Phi_{m})_{X_{n}}\\
 &O& &\vdots&\vdots&\vdots\\
&&&
(\Phi_{n})_{X_m}&\cdots&(\Phi_{n})_{X_{n}}
\end{array}
\right)(0)
=:
\left(
\begin{array}{c|c}
M_1&M_2\\
\hline
O&M_4
\end{array}
\right).
$$
Let us set
$$
\bar\lambda_i
=
\det\big(\Phi_1(f),\ldots,\Phi_{m-1}(f),
\Phi_{m-1+i}(f)\big)\quad
(i=1,\ldots,n-m+1).
$$
Then by a direct calculation we see
\begin{equation}\label{eq:lambdatar}
\begin{array}{L}
\left(\bar\lambda_i\right)_{x_k}(0)
=
\det M_1\left(
\sum_{j=1}^{n-m+1}(\Phi_{m-1+i})_{X_j}(\lambda_j)_{x_k}\right)(0)\quad
(k=1,\ldots,m),\\[6mm]
\trans{\big((\bar \lambda_1)_{x_k},\ldots,(\bar \lambda_{n-m+1})_{x_k}\big)}
\red{(0)}
=
\Big(\det M_1\ M_4
\trans{\big((\lambda_1)_{x_k},\ldots,(\lambda_{n-m+1})_{x_k}\big)}
\Big)\red{(0)}
\end{array}
\end{equation}
for any $i=1,\ldots,n-m+1$, where
$\trans{(\ )}$ is transposition.
Since $\det M_1\det M_4\ne0$ holds \red{at $0$}, 
this shows that non-degeneracy does not depend
on the choice of coordinate systems satisfying \eqref{eq:normal}.
We now show the second part.
It is easily seen that
$S(f)=\Lambda^{-1}(0)$
and non-degeneracy implies that $0$ is a regular value
of $\Lambda$.
Hence $S(f)$ is a manifold.
\end{proof}

Let 
$f:(\R^m,0)\to(\R^n,0)$ satisfies that $\rank df_0=m-1$.
Then there exists a vector field $\eta$ on $(\R^m,0)$
such that
$
\left\langle\eta_p\right\rangle_{\R}=\ker df_p
$
holds for $p\in S(f)$.
We call $\eta$ the {\em null vector field}.
In fact, since $\rank df_0=m-1$,
we may assume that $f$ is written as
\eqref{eq:normal}.
Since $S(f)=\{(f_m',\ldots,f_n')=0\}$ $({~}'=\partial/\partial x_m)$ holds,
$\partial x_m$ satisfies the condition of the null vector field.

Now we discuss higher order non-degeneracy 
and singularities,
by considering restriction of a map-germ
to its singular set.
The procedure is similar to that of the case
of the equidimensional Morin singularities given in \cite{aksing}.
Let $f:(\R^m,0)\to(\R^n,0)$ be a map-germ
and $0$ a non-degenerate singular point.
Let us assume that
$f=(f_1,\ldots,f_n)$ satisfies \eqref{eq:coord}.
Let $\lambda_i$ and $\Lambda$ be as in \eqref{eq:lambda}.
Since $S(f)$ is a manifold, the condition
$\eta_0\in T_0S(f)$ is well-defined.
A non-degenerate singular point $0$ of 
$f:(\R^m,0)\to(\R^n,0)$
is {\em $2$-singular\/}
if $\eta_0\in T_0S(f)$ holds.
This condition is equivalent to $\eta\Lambda(0)=0$,
where $\eta\Lambda$ stands for the directional derivative.
Set 
$S_2(f)=\{p\in S(f)\,|\,\eta_p\in T_pS(f)\}$.
The direction of $\eta$ is unique on $S(f)$,
the definition of $S_2(f)$ does not 
depend on the choice of $\eta$.
\red{Moreover, we have the following lemma.}
\begin{lemma}
\red{The equality\/ $S_2(f)=S(f|_{S(f)})$ holds.}
\end{lemma}
\begin{proof}
Since the conclusion does not depend on the choice of
coordinate systems and choice of $\eta$,
we may assume that $f$ is written in the form
\eqref{eq:normal},
and $\eta=\partial x_m$.
Transposition of the matrix representation of $d\Lambda_0$ is
\begin{equation}\label{eq:nondeg1}
\pmt{
(f_m')_{x_1}&\cdots&(f_n')_{x_1}\\
\vdots&\vdots&\vdots\\
(f_m')_{x_{m-1}}&\cdots&(f_n')_{x_{m-1}}\\
f_m''&\cdots&f_n''}(0)
=
\pmt{
(f_m')_{x_1}&\cdots&(f_n')_{x_1}\\
\vdots&\vdots&\vdots\\
(f_m')_{x_{m-1}}&\cdots&(f_n')_{x_{m-1}}\\
0&\cdots&0}(0),
\end{equation}
where $'=\partial/\partial x_m$.
We remark that
the assumption $2$-singular implies
$\eta\Lambda(0)=0$, thus the last row vanishes.
Since 
the rank of the matrix \eqref{eq:nondeg1} is $n-m+1$
by non-degeneracy,
we may assume
$$
\rank\pmt{
(f_m')_{x_1}&\cdots&(f_n')_{x_1}\\
\vdots&\vdots&\vdots\\
(f_m')_{x_{n-m+1}}&\cdots&(f_n')_{x_{n-m+1}}}(0)
=n-m+1
$$
by a numbering change.
By the implicit function theorem,
there exist functions
$$
x_1(x_{n-m+2},\ldots,x_{\red{m}}),\ldots,
x_{n-m+1}(x_{n-m+2},\ldots,x_{\red{m}})
$$
such that
\begin{equation}\label{eq:nondeg2}
\Lambda(x_1(x_{n-m+2},\ldots,x_m),\ldots,x_{n-m+1}(x_{n-m+2},\ldots,x_m),
x_{n-m+2},\ldots,x_m)\equiv0
\end{equation}
holds, where $\equiv$ means that the equality holds identically.
Differentiating \eqref{eq:nondeg2}
by $x_m$, we have
\begin{equation}\label{eq:sf1}
\pmt{
(\lambda_1)_{x_1}&\cdots&(\lambda_1)_{x_{n-m+1}}\\
\vdots&\vdots&\vdots\\
(\lambda_{n-m+1})_{x_1}&\cdots&(\lambda_{n-m+1})_{x_{n-m+1}}
}
\pmt{x_1'\\ \vdots \\ x_{n-m+1}'}
+
\pmt{\lambda_1'\\ \vdots \\ \lambda_{n-m+1}'}
\equiv0.
\end{equation}
On the other hand,
$g=f|_{S(f)}$ is parametrized by
$$
\begin{array}{L}
\Big(
x_1(x_{n-m+2},\ldots,x_{\red{m}}),\ldots,
x_{n-m+1}(x_{n-m+2},\ldots,x_{\red{m}}),
x_{n-m+2},\ldots,x_{m-1},\\
\hspace{8mm}
f_{m}\big(x_1(x_{n-m+2},\ldots,x_{\red{m}}),\ldots,
x_{n-m+1}(x_{n-m+2},\ldots,x_{\red{m}}),
x_{n-m+2},\ldots,x_{m}\big),\ldots,\\
\hspace{16mm}
f_{n}\big(x_1(x_{n-m+2},\ldots,x_{\red{m}}),\ldots,
x_{n-m+1}(x_{n-m+2},\ldots,x_{\red{m}}),
x_{n-m+2},\ldots,x_{m}\big)\Big).
\end{array}
$$
Since $f_m'=\cdots=f_n'=0$ on $S(f)$,
the transposition of the matrix representation of $dg$ is
\begin{equation}\label{eq:sf2}
\left.\left(
\begin{array}{ccc|c|CCC}
(x_1)_{x_{n-m+2}}&\cdots&(x_{n-m+1})_{x_{n-m+2}}& &*&\cdots&*\\
\vdots&\vdots&\vdots&E&\vdots&\vdots&\vdots\\
(x_1)_{x_{m-1}}&\cdots&(x_{n-m+1})_{x_{m-1}}& &*&\cdots&*\\
\hline
x_1'&\cdots&x_{n-m+1}'&0&
\sum_{i=1}^{n-m+1}(f_m)_{x_i}x_i'&\cdots&
\sum_{i=1}^{n-m+1}(f_n)_{x_i}x_i'
\end{array}
\right)\right|_{S(f)},
\end{equation}
where $E$ stands for the identity matrix.
Thus the matrix \eqref{eq:sf2} is not full-rank
if and only if $x_1'=\cdots=x_{n-m+1}'=0$.
By \eqref{eq:sf1}, the condition $x_1'=\cdots=x_{n-m+1}'=0$ 
is equivalent to
$(\eta\Lambda)|_{S(f)}=0$.
This implies that $S(g)=S_2(f)$.
\end{proof}

Let $0$ be a $2$-\red{singular} point of $f$.
We say that $0$ is {\em $2$-non-degenerate\/} if
$$d(\eta\Lambda)_0(T_0S(f))=T_0\R^{n-m+1}$$ holds.
This condition does not depend on the choice 
of coordinate systems and the choice of $\eta$.
Moreover, $2$-non-degeneracy implies that
$S_2(f)$ is a manifold.
In fact, it holds that
$S_2(f)=\{p\in S(f)\,|\,\eta_p\in T_pS(f)\}
=\{p\in S(f)\,|\,\eta\Lambda(p)=0\}$,
and that
$d(\eta\Lambda)_0(T_0S(f))=T_0\R^{n-m+1}$ implies
that $0$ is a regular value of 
$\eta\Lambda:S(f)\to\R^{\red{n-m+1}}$.
Since $\dim S(f)=m-(n-m+1)$ holds,
$m-(n-m+1)\geq n-m+1$ is needed for
$2$-non-degeneracy.
\begin{lemma}\label{lem:2nondeg}
Let\/ $f:(\R^m,0)\to(\R^n,0)$ be a map-germ with\/ $\rank df_0=m-1$.
A singular point\/ $0$ is $2$-non-degenerate
if and only if\/
$(\Lambda,\eta\Lambda)=0$, and
$$
\rank d(\Lambda,\eta\Lambda)_0=2(n-m+1).
$$
\end{lemma}
\begin{proof}
Let us assume 
$(\Lambda,\eta\Lambda)=0$ and $
\rank d(\Lambda,\eta\Lambda)_0=2(n-m+1)$.
Then we see that $\rank d\Lambda_0=n-m+1$, and
we see non-degeneracy.
Furthermore, $\eta\Lambda(0)=0$, so
we also see the $2$-singularity.
Thus it is enough to show that
$2$-non-degeneracy is equivalent to
$\red{\rank\,}d(\Lambda,\eta\Lambda)_0=2(n-m+1)$ at a $2$-singular point.

Let us assume that $0$ is $2$-singular.
\red{Since the dimension of $S(f)$ is $2m-n+1$,
and} by the $2$-singularity, it holds that
$\eta_0\in T_0S(f)$,
we take vector fields
$\xi_2,\ldots,\xi_{m}$ 
satisfying that $\eta,\xi_2,\ldots,\xi_{2m-n-\red{1}}$ at $0$ form
a basis of $T_0S(f)$.
By $S(f)=\{\Lambda=0\}$, 
it holds that
$\eta\Lambda=\xi_2\Lambda=\cdots=\xi_{2m-n-\red{1}}\Lambda=0$.
Thus
the transposition of the matrix representation 
of $d(\Lambda,\eta\Lambda)_0$ is
$$
\begin{array}{RCL}
&&
\left(
\begin{array}{ccc|ccc}
\eta\lambda_1&\cdots&\eta\lambda_{n-m+1}&
\eta^2\lambda_{1}&\cdots&\eta^2\lambda_{n-m+1}\\
\xi_2\lambda_1&\cdots&\xi_2\lambda_{n-m+1}&
\xi_2\eta\lambda_{1}&\cdots&\xi_2\eta\lambda_{n-m+1}\\
\vdots&\vdots&\vdots&\vdots&\vdots&\vdots\\
\xi_{2m-n-\red{1}}\lambda_1&\cdots&\xi_{2m-n-\red{1}}\lambda_{n-m+1}&
\xi_{2m-n-\red{1}}\eta\lambda_{1}&\cdots&\xi_{2m-n-\red{1}}
\eta\lambda_{n-m+1}\\
\hline
\xi_{\red{2m-n}}\lambda_1&\cdots&\xi_{\red{2m-n}}\lambda_{n-m+1}&
\xi_{\red{2m-n}}\eta\lambda_{1}&\cdots&\xi_{\red{2m-n}}\eta\lambda_{n-m+1}\\
\vdots&\vdots&\vdots&\vdots&\vdots&\vdots\\
\xi_{m}\lambda_1&\cdots&\xi_{m}\lambda_{n-m+1}&
\xi_{m}\eta\lambda_{1}&\cdots&\xi_{m}\eta\lambda_{n-m+1}
\end{array}
\right)(0)\\[23mm]
&=:&
\left(
\begin{array}{c|c}
O&J_2 \\
\hline
J_1&*
\end{array}
\right)(0).
\end{array}
$$
By the non-degeneracy, $\rank J_1(0)=n-m+1$ holds.
Hence $\rank d(\Lambda,\eta\Lambda)_0=2(n-m+1)$ is
equivalent to
$\rank J_2(0)=n-m+1$.
Since $\eta,\xi_2,\ldots,\xi_{2m-n-\red{1}}$ at $0$ form
a basis of $T_0S(f)$, we see that
$\rank J_2(0)=n-m+1$ is equivalent to $2$-non-degeneracy.
\end{proof}
Let $0$ be a $2$-non-degenerate singular point of $f$.
We say that $0$ is {\em $3$-singular\/} if
$\eta_0\in T_0S_2(f)$ holds, namely,
$\eta^2\Lambda(0)=0$,
where $\eta^j\Lambda$ stands for $\eta\cdots\eta\Lambda$ ($j$ times).
If $0$ is $3$-singular, we set
$S_3(f)=\{p\in S_2(f)\,|\,\eta_p\in T_pS_2(f)\}$.
This does not depend on the choice of $\eta$,
and it holds that
$S_3(f)=\{p\in S_2(f)\,|\,\eta^2\Lambda(p)=0\}
=\{p\in (\R^m,0)\,|\,\Lambda(p)=\eta\Lambda(p)=\eta^2\Lambda(p)=0\}$.

Accordingly, we define higher order singularities and non-degeneracies
inductively.
For a fixed $1\leq i\leq m/(n-m+1)$, and for $j\leq i-1$,
assume that
$j$-singularity and $j$-non-degeneracy
of a singular point $0$ of $f$
are defined, and
$S_j(f)
=\{p\in S_{j-1}(f)\,|\,\eta_p\in T_pS_{j-1}\}
=\{p\in S_{j-1}(f)\,|\,\eta^{j-1}\Lambda(p)=0\}$ and
$d(\eta^{j-1}\Lambda)_0(T_0S_{j-1}(f))=T_0\R^{n-m+1}$
holds.
This condition implies that $S_j(f)$ is a manifold.

Let $0$ be an $(i-1)$-non-degenerate singular point of $f$.
We say that $0$ is {\em $i$-singular\/} if
$\eta_0\in T_0S_{i-1}$ holds.
We define $S_i=\{p\in S_{i-1}\,|\,\eta_p\in T_pS_{i-1}\}$.
Then since 
$S_{i-1}(f)=\{p\in S_{i-2}(f)\,|\,\eta^{i-2}\Lambda(p)=0\}$,
we see that
$S_i=
\{p\in S_{i-1}\,|\,\eta^{i-1}\Lambda=0\}$.

Let $0$ be an $i$-singular point of $f$.
We call $0$ is {\em $i$-non-degenerate\/} if
$$d(\eta^{i-1}\Lambda)_0(T_0S_{i-1}(f))=T_0\R^{n-m+1}$$
holds.
We show the following lemma.
\begin{lemma}\label{lem:inondeg}
For an\/ $i$-singular point,
the\/ $i$-non-degeneracy does not depend 
on the choice of\/ $\eta$.
\end{lemma}
\begin{proof}
Let $\tilde \eta=\alpha\eta+\beta$,
where $\alpha$ is a non-zero function
and $\beta$ is a vector field satisfying
$\beta=0$ on $S(f)$.
It is enough to show that
$$
\tilde\eta^{i-1}\Lambda=\alpha^{i-1}\eta^{i-1}\Lambda
\quad(\text{on}\ S_{i-1}(f)).
$$
We show this by induction.
If $i=2$, it is obvious.
We assume that
$\tilde\eta^{i-2}\Lambda=\alpha^{i-2}\eta^{i-2}\Lambda$
holds on $S_{i-2}(f)$.
Then 
\begin{equation}\label{eq:inondeg}
\begin{array}{rcl}
\tilde\eta^{i-1}\Lambda-\alpha^{i-1}\eta^{i-1}\Lambda
&=&
(\alpha\eta+\beta)\tilde\eta^{i-2}\Lambda-\alpha^{i-1}\eta^{i-1}\Lambda\\
&=&
\alpha\Big(
\eta\underline{\big(\tilde\eta^{i-2}\Lambda-\alpha^{i-2}\eta^{i-2}\Lambda\big)}
+
\eta(\alpha^{i-2})\eta^{i-2}\Lambda\Big)
+
\beta\tilde\eta^{i-2}\Lambda
\end{array}
\end{equation}
holds.
Since the underlined part of \eqref{eq:inondeg} vanishes on $S_{i-2}(f)$, 
and $S_{i-1}(f)=\{\eta\in TS_{i-2}\}$,
and
$\eta^{i-2}\Lambda=0$ on $S_{i-1}$,
the right hand side of \eqref{eq:inondeg} vanishes
on $S_{i-1}(f)$.
\end{proof}
This procedure can be continued when
$i\leq m/(n-m+1)$.
We see that
$$S_i(f)=(\Lambda,\eta\Lambda,\ldots,\eta^{i-1}\Lambda)^{-1}(0).$$
More precisely, we have the following lemma.
\begin{lemma}\label{lem:key}
Let\/ $f:(\R^m,0)\to(\R^n,0)$ be a map-germ satisfying\/
 $\rank df_0=m-1$.
The\/ $(i+1)$-non-degeneracy of a singular point\/ $0$ is
equivalent to\/
$$
(\Lambda,\eta\Lambda,\ldots,\eta^i\Lambda)=0\quad \text{and}\quad
\rank d(\Lambda,\eta\Lambda,\ldots,\eta^i\Lambda)_0=(i+1)(n-m+1).
$$
\end{lemma}
\begin{proof}
We show the necessity by induction.
By Lemma \ref{lem:2nondeg}, we have $2$-non-degeneracy.
Let us assume that $j$-non-degeneracy ($j\leq i$) is proven.
The $(j+1)$-singularity of $0$ follows immediately from
$\eta^j\Lambda(0)=0$ for $j\leq i$.
We show $(j+1)$-non-degeneracy.
By $j$-non-degeneracy, we have submanifolds
$$
S_j\subset S_{j-1}\subset\cdots\subset S_1\subset (\R^m,0).
$$
We take a basis of each tangent space at $0$ as follows:
$\Xi_j=\{\xi_{j,1},\ldots,\xi_{j,m-j(n-m+1)}\}$
is a basis of $T_0S_j$,
$\Xi_{ k}=\{\xi_{k,1},\ldots,\xi_{k,n-m+1}\}\cup \Xi_{k+1}$
is a basis of $T_0S_{k}$ $(k=j-1,\ldots,1)$, and
$\Xi_{0}=\{\xi_{0,1},\ldots,\xi_{0,n-m+1}\}\cup \Xi_{1}$
is a basis of $T_0\R^m$.
Since $S_k(f)=\{\Lambda=\eta\Lambda=\cdots=\eta^{k-1}\Lambda=0\}$
$(1\leq k\leq j)$,
if $\xi\in T_0S_k(f)$, then $\xi\Lambda=\cdots\xi\eta^{k-1}\Lambda=0$
holds at $0$.
Thus
the transposition of the matrix representation 
of $d(\Lambda,\eta\Lambda,\ldots,\eta^j\Lambda)_0$ is
\newcommand{\fdm}[1]{\makebox[3em][c]{$#1$}}
$$
\begin{array}{cc}
&
\phantom{\Bigg(}
\begin{array}{ccccc}
\fdm{\Lambda}&\fdm{\eta\Lambda}&\fdm{\cdots}
&\fdm{\eta^{j-1}\Lambda}&\fdm{\eta^j\Lambda}\\
\vdots&\vdots&\cdots&\vdots&\vdots
\end{array}
\phantom{\Bigg)}
\\
\begin{array}{cc}
\Xi_j&\cdots\\
\Xi_{j-1}&\cdots\\
\vdots\\
\Xi_1&\cdots\\
\Xi_0&\cdots
\end{array}
&
\left(
\begin{array}{c|c|c|c|c}
\fdm{O}&\fdm{O}&\fdm{\cdots}&\fdm{O}&\fdm{J_j}\\
\hline
O&O&\cdots&J_{j-1}\\
\hline
\vdots&\vdots&\rotatebox{90}{$\ddots$}\\
\hline
O&J_2\\
\hline
J_1
\end{array}
\right),
\end{array}
$$
where
$$
\begin{array}{ccc}
&&\eta^k\Lambda\\
&&\vdots\\
\Xi_l&\cdots&A
\end{array}
$$
means that $A$ is a matrix formed by
differentials of $\eta^k\Lambda=\eta^k(\lambda_1,\ldots,\lambda_{n-m+1})$
by $\Xi_l=(\xi_{l,1},\ldots,\xi_{l,L})$.
Then we see that $\rank J_j=n-m+1$,
and this implies $(j+1)$-non-degeneracy.
\end{proof}
\begin{theorem}\label{thm:cri}
The map-germ\/ $f:(\R^m,0)\to(\R^n,0)$
is an\/ $r$-Morin singularity
if and only if\/ $0$ is\/ $r$-non-degenerate 
but not\/ $r$-singular.
\end{theorem}
To prove this theorem,
the assumption does not depend on
the choice of coordinate system and
choice of null vector field,
we may assume that $f$ is of the form
\begin{equation}\label{eq:normal2}
f
(x_1,\ldots,x_m)
=
\big(x_1,\ldots,x_{m-1},f_m(x_1,\ldots,x_m),\ldots,
f_n(x_1,\ldots,x_m)\big),
\end{equation}
and $\eta=\partial x_m$.
Then $\Lambda=(f_m',\ldots,f_n')(x_1,\ldots,x_m)$
holds,
where $'=\partial/\partial x_m$.
Then the theorem 
follows directly from \red{the following lemma due to
Morin.}
\begin{lemma}{\rm (Morin, \cite[p 5663, Lemme]{morin})}\label{lem:morin}
\red{Let\/ $f:(\R^m,0)\to(\R^n,0)$ is a map-germ written
in the form\/ \eqref{eq:normal2}.
Then\/ $f$ at\/ $0$ is an\/ $r$-Morin singularity
if and only if\/ 
$(f_m^{(j)},\ldots,f_n^{(j)})(0)=0$ $(1\leq j\leq r)$ and\/ 
$(f_m^{(r+1)},\ldots,f_n^{(r+1)})(0)\ne0$ hold,
and\/ 
$
\rank d(F, F',\ldots,F^{(r-1)})_{0}=r(n-m+1)
$ holds,
where\/  $F=(f_m',\ldots,f_n')$.
}
\end{lemma}
We give a proof of Theorem \ref{thm:cri}
here for the sake of those readers
who are not familiar with 
singularity theory.
\red{The proof is based on \cite[p 5664-5665]{morin}.
The proof is a little complicated, thus we 
would like to state a short sketch of it previously.
By the usual usage of the Malgrange preparation theorem and
by Tschirnhaus transformation,
one may assume that $f$ has the form
$$
\big(x_1,\ldots,x_{m-1},g_1(x),\ldots,
g_{n-m+1}(x)\big),
$$ where $x=(x_1,\ldots,x_m)$ and
\begin{equation}\label{eq:prfmorin}
g_i(x)=
\sum_{j=1}^r\tilde g_{ij}(x)x_m^j\ (i=1,\ldots,n-m),\quad
g_{n-m+1}(x)=
\sum_{j=1}^{r-1}\tilde g_{n-m+1,j}(x)x_m^j+x_m^{r+1}.
\end{equation}
If the coordinate change on the source 
$(x_1,\ldots,x_m)\mapsto(\tilde x_1,\ldots,\tilde x_m)$
defined by
$$
\begin{array}{l}
\tilde x_1=\tilde g_{11}(x),\ldots,
\tilde x_r=\tilde g_{1r}(x),
\tilde x_{r+1}=\tilde g_{21}(x),\ldots,
\tilde x_{2r}=\tilde g_{2r}(x),\ldots,\\
\hspace{20mm}
\tilde x_{r(n-m)+1}=\tilde g_{n-m+1,1}(x),\ldots,
\tilde x_{r(n-m)+r-1}=\tilde g_{n-m+1,r-1}(x),\\
\hspace{40mm}
\tilde x_{r(n-m)+r}=x_{r(n-m)+r},\ldots,
\tilde x_m=x_m
\end{array}
$$
is allowed,
then \eqref{eq:prfmorin} can be written in the form
$$
g_i(x)=
\sum_{j=1}^r\tilde x_{i-1+j}\tilde x_m^j\ (i=1,\ldots,n-m),\quad
g_{n-m+1}(x)=
\sum_{j=1}^{r-1}\tilde x_{n-m+j}(x)\tilde x_m^j+\tilde x_m^{r+1}.
$$
Then by a suitable coordinate change on the target,
the claim is proven.
Most of the proof is occupied to show
that these coordinate changes are regular.}
\begin{proof}[Proof of Theorem\/ {\rm \ref{thm:cri}}]
By $(r-1)$-singularity and non $r$-singularity,
$(f_m^{(j)},\ldots,f_n^{(j)})(0)=0$ $(1\leq j\leq r)$, and
$(f_m^{(r+1)},\ldots,f_n^{(r+1)})(0)\ne0$ holds.
By a linear transformation, 
we may assume $f_i^{(r+1)}(0)\ne0$ for all $m\leq i\leq n$.
We consider a quotient space
\begin{equation}\label{eq:morinring}
{\cal M}_m\big/\big\langle x_1,\ldots,x_{m-1},f_i(x)
\big\rangle_{{\cal M}_m}
=
\langle x_{m}^{r+1}\rangle_{{\cal M}_m}\quad
(m\leq i\leq n),
\end{equation}
where ${\cal M}_m=\{f:(\R^m,0)\to(\R,0)\}$ is a ring of 
function-germs.
Then by the Malgrange preparation theorem,
there exist functions
$\alpha_{n,k}$ $(0\leq k\leq r)$
such that
\begin{equation}\label{eq:morinpf1}
x_m^{r+1}
=
\alpha_{n,0}\big(x_1,\ldots,x_{m-1},f_n(x)\big)
-
\sum_{k=1}^r\alpha_{n,k}\big(x_1,\ldots,x_{m-1},f_n(x)\big)
x_m^k
\end{equation}
holds.
We consider a diffeomorphism-germ
$$
\psi(x_1,\ldots,x_m)=\left(
x_1,\ldots,x_{m-1},
x_m+\dfrac{1}{r}\alpha_{n,r}\big(x_1,\ldots,x_{m-1},f_n(x)\big)\right),
$$
and set
$\tilde x=\psi(x)$,
where 
$\tilde x=(\tilde x_1,\ldots,\tilde x_m)$ and
$x=(x_1,\ldots,x_m)$.
We remark that $\psi^{-1}$ has the form
$\psi^{-1}(\tilde x)
=\big(\tilde x_1,\ldots,\tilde x_{m-1},\psi_m^{-1}(\tilde x)\big)$.
Then by a calculation,
we see that there exist functions
$\beta_{n,k}$ $(0\leq k\leq r-1)$
such that
\begin{equation}\label{eq:morinpf2}
\tilde x_m^{r+1}
=
\beta_{n,0}\Big(
\tilde x_1,\ldots,\tilde x_{m-1},f_n\big(\psi^{-1}(\tilde x)\big)
\Big)
-
\sum_{k=1}^{r-1}
\beta_{n,k}\Big(
\tilde x_1,\ldots,\tilde x_{m-1},f_n\big(\psi^{-1}(\tilde x)\big)
\Big)
\tilde x_m^k
\end{equation}
holds.
Again by \eqref{eq:morinring}, there exist functions
$\beta_{i,k}$ $(0\leq k\leq r,\ m\leq i\leq n-1)$
such that
\begin{equation}\label{eq:morinpf3}
\tilde x_m^{r+1}
=
\beta_{i,0}\Big(
\tilde x_1,\ldots,\tilde x_{m-1},f_i\big(\psi^{-1}(\tilde x)\big)
\Big)
-
\sum_{k=1}^{r}
\beta_{i,k}\Big(
\tilde x_1,\ldots,\tilde x_{m-1},f_i\big(\psi^{-1}(\tilde x)\big)
\Big)
\tilde x_m^k
\end{equation}
holds.
Differentiating 
\eqref{eq:morinpf2} and \eqref{eq:morinpf3} $r+1$ times 
by $\tilde x_m$,
we see that 
$$\dfrac{\partial}{\partial y}
\beta_{n,0}(x_1,\ldots,x_{m-1},y)\ne0,\quad
\dfrac{\partial}{\partial y}
\beta_{i,0}(x_1,\ldots,x_{m-1},y)\ne0\quad
(m\leq i\leq n-1)
$$
at $0$.
Moreover, we have the following lemma.
\begin{lemma}\label{lem:morindiffeo}
Vectors 
$$
\begin{array}{l}
d\beta_{m,1}(x_1,\ldots,x_{m-1},0),\ldots,
d\beta_{m,r}(x_1,\ldots,x_{m-1},0),\\
d\beta_{m+1,1}(x_1,\ldots,x_{m-1},0),\ldots,
d\beta_{m+1,r}(x_1,\ldots,x_{m-1},0),\\
\hspace{10mm}\ldots,\\
d\beta_{n-1,1}(x_1,\ldots,x_{m-1},0),\ldots,
d\beta_{n-1,r}(x_1,\ldots,x_{m-1},0),\\
d\beta_{n,1}(x_1,\ldots,x_{m-1},0),\ldots,
d\beta_{n,r-1}(x_1,\ldots,x_{m-1},0)\\
\end{array}
$$
are linearly independent at\/ $0$.
\end{lemma}
\begin{proof}
Differentiating \eqref{eq:morinpf2} and \eqref{eq:morinpf3}
by $\tilde x_m$ and $\tilde x_l$ $(1\leq l\leq m-1)$,
we see that
$$
0=
(\beta_{i,0})_y(f_i)_{x_l}'(\psi_m^{-1})'-(\beta_{i,1})_{x_l}
\quad
(m\leq i\leq n)
$$
holds at $0$.
This implies that $d\beta_{i,1}(x_1,\ldots,x_{m-1},0)
=a_{i,11}df_i'(x_1,\ldots,x_{m-1},0)$ holds at $0$,
where $a_{i,11}\in\R$ is non-zero.
Again differentiating \eqref{eq:morinpf2} and \eqref{eq:morinpf3}
twice by $\tilde x_m$ and $\tilde x_l$ $(1\leq l\leq m-1)$,
we see that
$$
0=
(\beta_{i,0})_y(f_i)_{x_l}''((\psi_m^{-1})')^2
-(\beta_{i,1})_{y}(f_i)'_{x_l}(\psi_m^{-1})'
-2(\beta_{i,2})_{x_l}
\quad
(m\leq i\leq n)
$$
holds at $0$.
Thus it holds that
$$d\beta_{i,2}(x_1,\ldots,x_{m-1},0)
=a_{i,21}d(f_i')(x_1,\ldots,x_{m-1},0)
+a_{i,22}d(f_i'')(x_1,\ldots,x_{m-1},0)$$ at $0$,
where $a_{i,22},a_{i,21}\in\R$ and $a_{i,22}\ne0$.
By the same arguments,
we see that
$$
\begin{array}{L}
d\beta_{i,k}(x_1,\ldots,x_{m-1},0)
=\sum_{j=1}^ka_{i,kj}d\big(f_i^{(j)}\big)(x_1,\ldots,x_{m-1},0)
\ 
(1\leq k\leq r,\ m\leq i\leq n-1),\\
d\beta_{n,k}(x_1,\ldots,x_{m-1},0)
=\sum_{j=1}^ka_{n,kj}d\big(f_n^{(j)}\big)(x_1,\ldots,x_{m-1},0)
\ 
(1\leq k\leq r)
\end{array}
$$
at $0$,
where $a_{i,kk}\ne0,a_{n,kk}\ne0$.
This implies that
$$
\begin{array}{l}
\hspace{-10mm}\rank\big(
d\beta_{m,1}(x_1,\ldots,x_{m-1},0),\ldots,
d\beta_{m,r}(x_1,\ldots,x_{m-1},0),\\
d\beta_{m+1,1}(x_1,\ldots,x_{m-1},0),\ldots,
d\beta_{m+1,r}(x_1,\ldots,x_{m-1},0),\\
\hspace{10mm}\ldots,\\
d\beta_{n-1,1}(x_1,\ldots,x_{m-1},0),\ldots,
d\beta_{n-1,r}(x_1,\ldots,x_{m-1},0),\\
d\beta_{n,1}(x_1,\ldots,x_{m-1},0),\ldots,
d\beta_{n,r-1}(x_1,\ldots,x_{m-1},0)\big)
\end{array}
$$
is the same as
$$
\begin{array}{l}
\hspace{-10mm}\rank\big(
df_m'(x_1,\ldots,x_{m-1},0),\ldots,
df_m^{(r)}(x_1,\ldots,x_{m-1},0),\\
df_{m+1}'(x_1,\ldots,x_{m-1},0),\ldots,
df_{m+1}^{(r)}(x_1,\ldots,x_{m-1},0),\\
\hspace{10mm}\ldots,\\
df_{n-1}'(x_1,\ldots,x_{m-1},0),\ldots,
df_{n-1}^{(r)}(x_1,\ldots,x_{m-1},0),\\
df_{n}'(x_1,\ldots,x_{m-1},0),\ldots,
df_{n}^{(r-1)}(x_1,\ldots,x_{m-1},0)\big),
\end{array}
$$
and this is full-rank by assumption.
\end{proof}
Assume that
$$
\begin{array}{l}
\rank\big(
df_m'(x_1,\ldots,x_{r(n-m+1)-1},0,\ldots,0),\ldots,
df_m^{(r)}(x_1,\ldots,x_{r(n-m+1)-1},0,\ldots,0),\\
\hspace{5mm}df_{m+1}'(x_1,\ldots,x_{r(n-m+1)-1},0,\ldots,0),\ldots,
df_{m+1}^{(r)}(x_1,\ldots,x_{r(n-m+1)-1},0,\ldots,0),\\
\hspace{10mm}\ldots,\\
\hspace{5mm}df_{n-1}'(x_1,\ldots,x_{r(n-m+1)-1},0,\ldots,0),\ldots,
df_{n-1}^{(r)}(x_1,\ldots,x_{r(n-m+1)-1},0,\ldots,0),\\
\hspace{5mm}df_{n}'(x_1,\ldots,x_{r(n-m+1)-1},0,\ldots,0),\ldots,
df_{n}^{(r-1)}(x_1,\ldots,x_{r(n-m+1)-1},0,\ldots,0)\big)\\
={r(n-m+1)-1}.
\end{array}
$$
Then the map $\theta$ defined by
\begin{equation}\label{eq:diff1}
\begin{array}{rL}
\tilde x\mapsto&
\bigg(
\beta_{m,1}
\Big(\tilde x_1,\ldots,\tilde x_{m-1},f_m\big(\psi^{-1}(\tilde x)\big)
\Big),
\ldots,
\beta_{m,r}\Big(\tilde x_1,\ldots,\tilde x_{m-1},f_m\big(\psi^{-1}(\tilde x)\big)
\Big),\\
&\hspace{0mm}
\beta_{m+1,1}
\Big(\tilde x_1,\ldots,\tilde x_{m-1},f_{m+1}\big(\psi^{-1}(\tilde x)\big)
\Big),
\ldots,
\beta_{m+1,r}\Big(\tilde x_1,\ldots,\tilde x_{m-1},f_{m+1}\big(\psi^{-1}(\tilde x)\big)
\Big),\\
&\hspace{10mm},\ldots,\\
&\hspace{0mm}
\beta_{n-1,1}
\Big(\tilde x_1,\ldots,\tilde x_{m-1},f_{n-1}\big(\psi^{-1}(\tilde x)\big)
\Big),
\ldots,
\beta_{n-1,r}\Big(\tilde x_1,\ldots,\tilde x_{m-1},
f_{n-1}\big(\psi^{-1}(\tilde x)\big)
\Big),\\
&\hspace{0mm}
\beta_{n,1}
\Big(\tilde x_1,\ldots,\tilde x_{m-1},f_n\big(\psi^{-1}(\tilde x)\big)
\Big),
\ldots,
\beta_{n,r-1}\Big(\tilde x_1,\ldots,\tilde x_{m-1},f_n\big(\psi^{-1}(\tilde x)\big)
\Big),\\
&\hspace{3mm}
\tilde x_{r(n-m+1)},\ldots,\tilde x_m\bigg)
\end{array}
\end{equation}
is a diffeomorphism-germ
on the source, and $\Theta$ defined by
\begin{equation}\label{eq:diff2}
\begin{array}{rL}
X\mapsto&
\big(
\beta_{m,1}(X_1,\ldots,X_{m-1},X_m),\ldots,
\beta_{m,r}(X_1,\ldots,X_{m-1},X_m),\\
&\hspace{2mm}
\beta_{m+1,1}(X_1,\ldots,X_{m-1},X_{m+1}),\ldots,
\beta_{m+1,r}(X_1,\ldots,X_{m-1},X_{m+1}),\\
&\hspace{10mm}\ldots,\\
&\hspace{2mm}
\beta_{n-1,1}(X_1,\ldots,X_{m-1},X_{n-1}),\ldots,
\beta_{n-1,r}(X_1,\ldots,X_{m-1},X_{n-1}),\\
&\hspace{2mm}
\beta_{n,1}(X_1,\ldots,X_{m-1},X_{n}),\ldots,
\beta_{n,r-1}(X_1,\ldots,X_{m-1},X_{n}),\\
&\hspace{2mm}
X_{r(n-m+1)},\ldots,X_{m-1},\beta_{m,0}(X_1,\ldots,X_{m-1},X_m),\ldots,
\beta_{n,0}(X_1,\ldots,X_{m-1},X_n)\big),
\end{array}
\end{equation}
where $X=(X_1,\ldots,X_n)$,
is also a diffeomorphism-germ on the target.
We set $\theta(x)=\bar x=(\bar x_1,\ldots,\bar x_m)$.
Then we see that
$\Theta\circ f\circ\psi^{-1}\circ\theta^{-1}$
has the following expression:
$$
\begin{array}{rcL}
\beta_{i,j}\big(f\circ\psi^{-1}\circ\theta^{-1}(\bar x)\big)
&=&
\beta_{i,j}\Big(\bar x_1,\ldots,\bar x_{m-1},
f_{i}\big(\psi^{-1}\circ\theta^{-1}(\bar x)\big)\Big)\\[2mm]
&=&
\beta_{i,j}\Big(\tilde x_1,\ldots,\tilde x_{m-1},
f_{i}\big(\psi^{-1}(\tilde x)\big)\Big)\\
&=&
\bar x_{r(i-m)+j}
\end{array}
$$
when
$m\leq i\leq n-1,\ 1\leq j\leq r$
or
$i=n,\ 1\leq j\leq r-1$,
and
$$
\begin{array}{rcL}
\beta_{i,0}\big(f\circ\psi^{-1}\circ\theta^{-1}(\bar x)\big)
&=&
\beta_{i,0}\Big(\bar x_1,\ldots,\bar x_{m-1},
f_{i}\big(\psi^{-1}\circ\theta^{-1}(\bar x)\big)\Big)\\[2mm]
&=&
\beta_{i,0}\Big(\tilde x_1,\ldots,\tilde x_{m-1},
f_{i}\big(\psi^{-1}(\tilde x)\big)\Big)\\[2mm]
&=&
\tilde x_m^{r+1}+
\sum_{j=1}^R\beta_{i,j}\Big(\tilde x_1,\ldots,\tilde x_{m-1},
f_{i}\big(\psi^{-1}(\tilde x)\big)\Big)\tilde x_m^j\\
&=&
\bar x_m^{r+1}+
\sum_{j=1}^R\bar x_{r(i-m)+j}\bar x_m^j,
\end{array}
$$
where $R=r$ for $m\leq i\leq n-1$ and $R=r-1$ for $i=n$.
Therefore $f$ is $\A$-equivalent to
$$
x\mapsto
\big(x_1,\ldots,x_{m-1}, \hat h_m(x),\ldots, \hat h_{m-1}(x),h_n(x)\big),
$$
where $\hat h_{i}(x)=h_i(x)+x_m^{r+1}$, and $h_i(x)$ $(i=m,\ldots,n)$
are as in \eqref{eq:morinnor2}.
By suitable linear translations on the source and target, we see that
$f$ is  $\A$-equivalent to the form as in \eqref{eq:morinnor1}.
This completes the proof.
\end{proof}
By Theorem \ref{thm:cri} and Lemma \ref{lem:key},
we have the following criteria.
\begin{corollary}\label{cor:cri}
Let\/ $f:(\R^m,0)\to(\R^n,0)$ be a map-germ
satisfying\/ $\rank df_0=m-1$.
Then\/ $f$ at\/ $0$ is an\/ $r$-Morin singularity
if and only if
\begin{itemize}
\item
$\eta\Lambda=\cdots=\eta^{r-1}\Lambda=0$ and $\eta^{r}\Lambda\ne0$
hold at\/ $0$, and
\item
$\rank d(\Lambda,\eta\Lambda,\ldots,\eta^{r-1}\Lambda)_0=r(n-m+1)$
holds.
\end{itemize}
Here, \hfill
$f=(f_1,\ldots,f_{m})$\hfill satisfies\hfill that\/\hfill
$d(f_1,\ldots,f_{m-1})=m-1$,\hfill
$\Lambda=(\lambda_1,\ldots,\lambda_{n-m+1})$,\\
$\lambda_i=\det(f_1,\ldots,f_{m-1},f_{m-1+i})$
and\/
$\eta$ is the null vector field.
\end{corollary}
\red{Applying
Lemma \ref{lem:morin} for a given map-germ $f$,
it needs that $f$ is written in the normalized 
form \eqref{eq:normal2}, and to obtain this form, 
the implicit function theorem is applied.
On the other hand, since
our criteria uses only coordinate free data of $f$,
the author believes that our criteria (Theorem \ref{thm:cri}
and Corollary \ref{cor:cri}) is convenient to 
Lemma \ref{lem:morin} and indispensable in certain cases.
In fact, applications \cite{ishimachi,ist,kruy,nishi,syy} 
of this kind of criteria
might be difficult by using only of the criteria
which needs the normalization.}
We remark that 
our characterization can be interpreted as
a vector field representation of the intrinsic 
derivative.
See \cite{porteous} about the intrinsic derivative,
and see also \cite{ando,GG}.
In fact, the image of $v\in T_p\R^m$ by
$D(df)_p:T_p\R^m\to \Hom(K_p,L_{f(p)})$
coincides with 
$d\Lambda_p(v):\R\to\R^{n-m+1}$,
where $K_p=\ker df_p$, $L_{f(p)}=\coker df_p$,
and $T_p\R^k$ (resp. $T_p\Hom(K_p,L_{f(p)})$) is canonically
identified with $\R^k$ ($k=1,n-m+1$) (resp. $\Hom(K_p,L_{f(p)})$).

\section{Application to singularities of ruling maps}
A {\em one-parameter family of\/ $n$-planes in\/ $\R^{2n}$}
is a map defined by
$$
F_{(\gamma,\delta)}(t,u_1,\ldots,u_n)=\gamma(t)+\sum_{i=1}^nu_i\delta_i(t)
$$
where $\gamma:J\to\R^{2n}$ is a curve
and $\delta(t)=(\delta_1(t),\ldots,\delta_n(t)):J\to (\R^{2n})^n$
satisfies 
$\delta_i\cdot\delta_j=1$ if $i=j$ and
$\delta_i\cdot\delta_j=0$ if $i\ne j$,
where 
$J$ is an open interval, and
$\cdot$ stands for the canonical inner product.
We call $\gamma$ the {\em base curve},
and $\delta$ the {\em director frame}\/ of $F_{(\gamma,\delta)}$.
This is a generalization of ruled surfaces in $\R^3$.
Ruled surfaces are classical objects in differential geometry.
However, it has again paid attention in 
several areas \cite{pw,sasaki,sch}.
In general, ruled surfaces and their generalizations
have singularities, and 
they have been investigated 
in several articles \cite{ishidev, iztake, mu}.
To study the geometry and singularities of this kind of map,
the striction curve plays a crucial role 
(See \cite{gray,iztake}, for example).
One can always choose a director frame 
satisfying $\delta_i\cdot\delta_j'=0$ for any $i,j$.
A curve $\sigma(t)=\gamma(t)+\sum_{i=1}^nu_i(t)\delta_i(t)$
is a {\em striction curve}
if $\sigma'\cdot\delta_i'\equiv0$ $(1\leq i\leq n)$ holds,
where $\equiv$ means that the equality holds
identically.
If
$(\delta(t),\delta'(t))=
(\delta_1(t),\ldots,\delta_n(t),\delta_1'(t),\ldots,\delta_n'(t))$
are linearly independent, 
then we obtain 
a striction curve $\sigma(t)=\gamma(t)+\sum_{i=1}^nu_i(t)\delta_i(t)$
by setting 
$$
\pmt{u_1(t)\\ \vdots \\ u_n(t)}
=
-\Big(\big(\delta_i'(t)\cdot\delta_j'(t)\big)_{i,j=1,\ldots,n}\Big)^{-1}
\pmt{\gamma'\cdot \delta_1'\\ \vdots \\ \gamma'\cdot \delta_n'}(t).
$$
One can easily show that
the image of the striction curve coincides with
the set of singular points of $F_{(\gamma,\delta)}$.
Moreover, 
$p=(t,u_1,\ldots,u_n)$ is a $1$-Morin singularity
if and only if the striction curve is an immersion at $p$ 
(\cite[Theorem 2.5]{sajihiro1} and \cite[Theorem 4]{sajibcp}).
We give an alternative proof of this fact 
by using our criteria.
\begin{proof}
Let $
F_{(\gamma,\delta)}$
be a one-parameter family of $n$-planes in $\R^{2n}$.
We assume that for any $t$,
$(\delta(t),\delta'(t))$ are linearly independent, 
$\delta_i\cdot\delta'_j=0$ $(i,j=1,\ldots,n)$,
and $\gamma$ is a striction curve.
Then $S(F_{(\gamma,\delta)})=\{u_1=\cdots=u_n=0\}$.
By the definition of striction curve,
there exist $\alpha_i(t)$ $(1\leq i\leq n)$ such that
$\gamma'(t)=\sum_{i=1}^n\alpha_i(t)\delta_i(t)$ holds.
Hence we see that
the null vector field $\eta$ can be taken as a function of $t$
and
$\eta(t)=-\partial t+\sum_{i=1}^n\alpha_i(t)\partial u_i$.
Moreover, since
$(t,u_1,\ldots,u_n)$ and
the coordinate system generated by $(\delta,\delta')(0)$
satisfies the condition \eqref{eq:coord},
$\Lambda=(\lambda_1,\ldots,\lambda_n)$ is
$$
\lambda_j
=\det\left(
\gamma'+\sum_{i=1}^nu_i\delta_i',\delta,
\delta'_1,\ldots,\widehat{\delta'_j},\ldots,\delta'_n\right)
=
\det\left(
\gamma'+u_j\delta_j',\delta,
\delta'_1,\ldots,\widehat{\delta'_j},\ldots,\delta'_n\right),\\
$$
where $\delta=(\delta_1,\ldots,\delta_n)$
and
$
(\delta'_1,\ldots,\widehat{\delta'_j},\ldots,\delta'_n)
=
(\delta'_1,\ldots,\delta'_{j-1},\delta'_{j+1},\ldots,\delta'_n)
$.
Then by Corollary \ref{cor:cri},
$F_{(\gamma,\delta)}$ at $p=(t,0,\ldots,0)$ is a $1$-Morin singularity
if and only if
$\eta\Lambda\ne0$.
By a direct calculation,
$$
\begin{array}{rcL}
\eta\lambda_j(p)&=&
-\det\left(
\gamma'+u_j\delta_j',\delta,
\delta'_1,\ldots,\widehat{\delta'_j},\ldots,\delta'_n\right)'\Big|_{u_j=0}\\
&&\hspace{5mm}
+
\det\left(\alpha_j
\delta_j',\delta,
\delta'_1,\ldots,\widehat{\delta'_j},\ldots,\delta'_n\right)(t)\\
&=&
-\det\left(
\gamma'',\delta,
\delta'_1,\ldots,\widehat{\delta'_j},\ldots,\delta'_n\right)(t)\\
&&\hspace{5mm}-\det\left(
\gamma',\delta_1,\ldots,\delta_j',\ldots,\delta_n,
\delta'_1,\ldots,\widehat{\delta'_j},\ldots,\delta'_n\right)(t)
+(-1)^{n+j-1}\alpha_j\Delta\\
&=&-\det\left(
\alpha_j\delta_j',\delta,
\delta'_1,\ldots,\widehat{\delta'_j},\ldots,\delta'_n\right)(t)\\
&&\hspace{5mm}-\det\left(
\alpha_j\delta_j,\delta_1,\ldots,\delta_j',\ldots,\delta_n,
\delta'_1,\ldots,\widehat{\delta'_j},\ldots,\delta'_n\right)(t)
+(-1)^{n+j-1}\alpha_j\Delta\\
&=&(-1)^{n+j}\alpha_j\Delta+(-1)^{n+j-1}\alpha_j\Delta
+(-1)^{n+j-1}\alpha_j\Delta\\
&=&(-1)^{n+j-1}\alpha_j\Delta,
\end{array}
$$
where $\Delta=\det(\delta,\delta')$.
Hence $\eta\Lambda\ne0$ is equivalent to $(\alpha_1,\ldots,\alpha_n)\ne0$,
and it is equivalent to $\gamma'\ne0$.
\end{proof}
\section{${\cal A}$-isotopy of map-germs}
We define an equivalence relation 
called {\em $\A$-isotopy}, which is a 
strengthened version of $\A$-equivalence.
Let $d$ be a natural number.
A map-germ $f\in C^\infty(m,n)$ is said to be
{\em $d$-determined}\/ if 
any $g\in C^\infty(m,n)$ satisfying
$j^df(0)=j^dg(0)$ is
$\A$-equivalent to $f$,
where $j^df(0)$ is the $d$-jet of $f$ at $0$.
Let $\diff{k}{d}$ be the set of $d$-jets of
diffeomorphism-germs
$(\R^k,0)\to(\R^k,0)$ equipped with the relative
topology as a subset $\diff{k}{d}\subset J^d(k,k)$,
where $J^d(k,k)$ is canonically identified 
with a Euclidean space.
\begin{definition}
Let $f,g\in C^\infty(m,n)$ be $\A$-equivalent
map-germs that are $d$-determined.
Then $f$ and $g$ are {\em $\A$-isotopic}\/
if there exist continuous curves
$\sigma:[0,1]\to\diff{m}{d}$ and
$\tau:[0,1]\to\diff{n}{d}$ such that
$\sigma(0)$, $\tau(0)$ are both $d$-jets of the identity, and
$$j^d(g)(0)=j^d\big(\tau(1)\circ f\circ \sigma(1)\big)(0)$$
holds.
\end{definition}
Namely, $f$ and $g$ are $\A$-isotopic
if and only if
$j^df(0)$ and $j^d g(0)$ are located on
the same arc-wise connected component of
the $d$-jet of the $\A^d$-orbit of $j^df(0)$.
Since the set $\diff{m}{d,+}$ of $d$-jets
of orientation-preserving diffeomorphism-germs
is arc-wise connected,
$f$ and $g$ are $\A$-isotopic
if and only if
there exist
orientation preserving diffeomorphism-germs
$\sigma^+:(\R^m,0)\to(\R^m,0)$ and
$\tau^+:(\R^n,0)\to(\R^n,0)$
such that
$j^dg(0)=j^d(\tau^+\circ f\circ \sigma^+)(0)$ holds.
This notion of $\A$-isotopic is a slightly strengthened 
version of $\A$-equivalence. By the above arguments,
there are at most four $\A$-isotopy classes in an
$\A$-equivalent class.
However, the number of $\A$-isotopy classes of
an $\A$-equivalent class of a given map-germ $f$
may represent a property of $f$.
In this section, we study the number of $\A$-isotopy 
classes of each Morin singularity as an application of 
our criteria (Corollary \ref{cor:cri}).
We remark that this problem was first asked by Takashi Nishimura
\cite[p.226]{nfbook} as far as the author knows.

It is easy to see that any corank $1$ germ
is $\A$-isotopic to the form 
\eqref{eq:normal2}.
Furthermore, since we only used the diffeomorphisms
\eqref{eq:diff1} and \eqref{eq:diff2}
to obtain
the normal form \eqref{eq:morinnor1} from \eqref{eq:normal2},
any $r$-Morin singularity
is $\A$-isotopic to
\begin{equation}\label{eq:moriniso}
\begin{array}{L}
h_{r,(\ep_1,\ep_2)}(x)=\\
\hspace{3mm}
\left(\ep_1x_1,x_2,\ldots,x_{m-1},
\ep_1x_{1}x_m+\sum_{j=2}^rx_{j}x_m^j
,h_2(x)\red{,}\ldots,
h_{n-m}(x),\ep_2h_{n-m+1}(x)\right),
\end{array}
\end{equation}
where $\ep_1=\pm1,\ep_2=\pm1$, and $h_2,\ldots,h_{n-m+1}$ are as in 
\eqref{eq:morinnor1}.
We remark that the final linear translations are orientation-preserving.
We have the following.
\begin{proposition}\label{prop:iso1}
{\rm (I)} If\/ $r$ is even, 
then\/ $h_{r,(\ep_1,\ep_2)}$ is\/
$\A$-isotopic to\/ $h_{r,(\ep_1,1)}$.
Moreover, if\/ $m>r(m-n+1)$ holds,
then\/ $h_{r,(\ep_1,\ep_2)}$ is\/
$\A$-isotopic to\/ $h_{0,r}$.
{\rm (II)} If\/ $r$ is odd, 
then\/ $h_{r,(\ep_1,\ep_2)}$ is\/
$\A$-isotopic to\/ $h_{r,(1,\ep_2)}$.
Moreover, if\/ $m>r(m-n+1)$ holds,
then\/ $h_{r,(\ep_1,\ep_2)}$ is\/
$\A$-isotopic to\/ $h_{0,r}$.
\end{proposition}
The proof of this proposition is not difficult,
but rather long. We postpone it to
Section \ref{sec:prf}.
By Proposition \ref{prop:iso1}, the $\A$-isotopic condition for
$r$-Morin singularities of suspensions ($m>r(n-m+1)$)
is the same as $\A$-equivalence, so
we stick to the non-suspension case ($m=r(n-m+1)$).
In this case,
by Corollary \ref{cor:cri},
a necessary condition that
$f$ is $\A$-equivalent to an $r$-Morin singularity
is 
\begin{equation}\label{eq:d}
\det d(\Lambda,\Lambda',\ldots,\Lambda^{(r-1)})(0)\ne0.
\end{equation}
Set $D=\sign \det d(\Lambda,\Lambda',\ldots,\Lambda^{(r-1)})(0)$,
and $a=n-m$.
Calculating $D$ for \eqref{eq:moriniso},
we obtain $D=\ep_1^{(a+1)r+1}\ep_2^r$.
Furthermore, the sign of $D$ may depend on 
the choice of oriented frame $\{\xi_1,\ldots,\xi_{m-1},\eta\}$,
and
an orientation-preserving diffeomorphism on the target.
Let $\{\tilde\xi_1,\ldots,\tilde\xi_{m-1},\tilde\eta\}$
be another frame, and let $\tilde D$ stand for the
sign of \eqref{eq:d} with respect to this
frame. Then $\tilde\eta(0)=\alpha \eta(0)$ holds.
If $\alpha>0$ then $\tilde D=D$, and 
if $\alpha<0$, then $\tilde D=(-1)^{(r-1)r(a+1)/2}D$
holds.
On the other hand,
let $\Phi=(\Phi_1,\ldots,\Phi_n)$ be
an orientation-preserving diffeomorphism on the target,
and let
$\bar D$ stand for the sign of 
\eqref{eq:d} of $\Phi\circ f$.
By \eqref{eq:lambdatar},
if $(\Phi_1,\ldots,\Phi_{m-1})|_{\{x_m=\cdots=x_n=0\}}$
is orientation-preserving, then $\Bar D=D$,
and
if $(\Phi_1,\ldots,\Phi_{m-1})|_{\{x_m=\cdots=x_n=0\}}$
is orientation-reversing, then $\Bar D=(-1)^{ar}D$
holds.
We divide $r$ into four cases via modulo four.
Let $l$ be an integer.

\underline{Case 1: $r=4l$}
In this case, 
$h_{r,(\ep_1,\ep_2)}$ and $h_{r,(\ep_1',\ep_2')}$ are 
$\A$-isotopic if and only if $\ep_1=\ep_1'$.
\begin{proof}
By Proposition \ref{prop:iso1},
$\ep_2$ may be deleted.
By the above arguments,
$D=\ep_1$ is an invariant of the $\A$-isotopic condition.
\end{proof}

\underline{Case 2: $r=4l+1$}
In this case, 
if $a$ is even, then 
$h_{r,(\ep_1,\ep_2)}$ and $h_{r,(\ep_1',\ep_2')}$ are 
$\A$-isotopic if and only if $\ep_2=\ep_2'$.
If $a$ is odd, then 
$h_{r,(\ep_1,\ep_2)}$ is
$\A$-isotopic to $h_{0,r}$.
\begin{proof}
By Proposition \ref{prop:iso1}, $\ep_1$ may be deleted.
By the above arguments again,
$D=\ep_2$ is an invariant of the $\A$-isotopic condition, and
we have the first conclusion.
For a proof of the second conclusion, see Section \ref{sec:prf}.
\end{proof}
In particular, the $\A$-class and the $\A$-isotopy class
coincides for the Whitney umbrella ($m=2,n=3,r=1$).

\underline{Case 3: $r=4l+2$}
In this case, 
if $a$ is odd, then 
$h_{r,(\ep_1,\ep_2)}$ and $h_{r,(\ep_1',\ep_2')}$ are 
$\A$-isotopic if and only if $\ep_1=\ep_1'$.
If $a$ is even, then 
$h_{r,(\ep_1,\ep_2)}$ is
$\A$-isotopic to $h_{0,r}$.
\begin{proof}
By Proposition \ref{prop:iso1},  $\ep_2$ may be deleted.
By the above arguments again,
$D=\ep_1$ is an invariant of the $\A$-isotopic condition, and
we have the first conclusion.
For a proof of the second conclusion, see Section \ref{sec:prf}.
\end{proof}
\underline{Case 4: $r=4l+3$}
In this case,
$h_{r,(\ep_1,\ep_2)}$ is
$\A$-isotopic to $h_{0,r}$.
See Section \ref{sec:prf} for a proof.

Summarizing up the above arguments,
we can summarize 
the number of $\A$-isotopy classes
of $\A$-classes for each Morin singularity.
We summarize it in the following table.
\begin{table}[!h]
\centering
\begin{tabular}{|c|c|c|c|}
\hline
&
\multicolumn{2}{c|}{$m=r(n-m+1)$}&
$m>r(n-m+1)$\\
\cline{1-3}
&$a$ : odd (invariant)&$a$ : even (invariant)&\\
\hline
\hline
$r=4l$&  2 $(\ep_1)$&2 $(\ep_1)$&1\\
\hline
$r=4l+1$&1&2 $(\ep_2)$&1\\
\hline
$r=4l+2$&2 $(\ep_1)$&1&1\\
\hline
$r=4l+3$&1&1&1\\
\hline
\end{tabular}
\caption{Number of $\A$-isotopy classes
in the $\A$-classes.}
\end{table}

\section{Proofs}\label{sec:prf}
Here, we use the following terminology:
Let $I$ be a set of indices such that
$\#I$ is even.
Then the {\em $\pi$-rotations of\/ $I$\/} are diffeomorphisms
$(x_1,\ldots,x_{k})\mapsto
(\tilde x_1,\ldots,\tilde x_{k}),$
where $\tilde x_j=\ep x_j$ if $j\in I$,
and $\tilde x_j=x_j$ if $j\not\in I$, with $\ep=-1$.
We see that 
applying 
$\pi$-rotations both on the source and the target
does not change the $\A$-isotopy class.
\begin{proof}[Proof of Proposition $\ref{prop:iso1}$]
First we show (I).
Set $\ep_2=-1$.
By a $\pi$-rotation of $\{m,n\}$ on the target,
$h_{r,(\ep_1,\ep_2)}$ is
$\A$-isotopic to
\begin{equation}\label{eq:prfiso0100}
\left(\ep_1x_1,\ldots,x_{m-1},
\ep_2\left(\ep_1x_{1}x_m+\sum_{j=2}^rx_{j}x_m^j\right)
,h_2(x)\red{,}\ldots,
h_{n-m}(x),h_{n-m+1}(x)\right).
\end{equation}
Considering $\pi$-rotations of $\{1,\ldots,r\}$ on the source,
\eqref{eq:prfiso0100} is
$\A$-isotopic to 
\begin{equation}\label{eq:prfiso0200}
\left(\ep_1\ep_2x_1,\ep_2x_2\red{,}\ldots,\ep_2x_r,x_{r+1},\ldots,x_{m-1},
\ep_1x_{1}x_m+\sum_{j=2}^rx_{j}x_m^j
,h_2(x)\red{,}\ldots,
h_{n-m+1}(x)\right).
\end{equation}
Considering $\pi$-rotations of $\{1,\ldots,r\}$ on the target,
\eqref{eq:prfiso0200} is
$\A$-isotopic to 
\begin{equation}\label{eq:prfiso0250}
\left(\ep_1x_1,\ldots,x_{m-1},
\ep_1x_{1}x_m+\sum_{j=2}^rx_{j}x_m^j
,h_2(x)\red{,}\ldots,
h_{n-m+1}(x)\right),
\end{equation}
which proves the first part of (I).
We assume $m>r(m-n+1)$ and set $\ep_1=-1$.
Then
$x_{m-1}$ is not contained in any terms of $h_1,\ldots,h_{n-m+1}$.
Considering $\pi$-rotations of $\{2,\ldots,r,m-1\}$ on the source,
\eqref{eq:prfiso0250} is
$\A$-isotopic to 
\begin{equation}\label{eq:prfiso0300}
\big(\ep_1x_1,\ldots,\ep_1x_r,x_{r+1},\ldots,x_{m-2},\ep_1x_{m-1},
\ep_1h_1(x),
h_2(x)\red{,}\ldots\big),
\end{equation}
where $h_1$ is as in \eqref{eq:morinnor1}.
Considering $\pi$-rotations of $\{1,\ldots,r,m-1,m\}$ on the target,
\eqref{eq:prfiso0300} is
$\A$-isotopic to $h_{0,r}$,
which proves the second part of (I).

Secondly, we show (II).
Set $\ep_1=-1$.
Considering $\pi$-rotations of $\{2,\ldots,r\}$ on the source,
\eqref{eq:prfiso0100} is
$\A$-isotopic to 
\begin{equation}\label{eq:prfiso0400}
\big(\ep_1x_1,\ldots,\ep_1x_r,x_{r+1},\ldots,x_{m-1},
\ep_1h_1(x)
,h_2(x)\red{,}\ldots,
h_{n-m}(x),
\ep_2h_{n-m+1}(x)\big).
\end{equation}
Then by $\pi$-rotations on the target, we see that
\eqref{eq:prfiso0400} is $\A$-isotopic to 
$h_{r,(1,\ep_2)}$, which proves the first part
of (II).
We assume $m>r(m-n+1)$ and set $\ep_2=-1$.
Then by $\pi$-rotations on the target, 
$h_{r,(1,\ep_2)}$ is $\A$-isotopic to 
\begin{equation}\label{eq:prfiso0500}
\big(x_1,\ldots,x_{m-1},
\ep_2h_1(x),
h_2(x)\red{,}\ldots,h_{n-m+1}(x)\big).
\end{equation}
Considering $\pi$-rotations of $\{1,\ldots,r,m-1\}$ of the source,
\eqref{eq:prfiso0500} is
$\A$-isotopic to 
\begin{equation}\label{eq:prfiso0600}
\big(\ep_2x_1,\ldots,\ep_2x_r,x_{r+1},\ldots,x_{m-2},\ep_2x_{m-1},
h_1(x),
h_2(x)\red{,}\ldots,h_{n-m+1}(x)\big).
\end{equation}
Then by $\pi$-rotations on the target, we see that
\eqref{eq:prfiso0600} is $\A$-isotopic to 
$h_{0,r}$, which proves the second part
of (II).
\end{proof}
\begin{proof}[Proof of the claim of the second part of Case\/ $2$]
Let us assume $r=4l+1$ and $a$ is odd.
By Proposition \ref{prop:iso1},
$h_{r,(\ep_1,\ep_2)}$ is $\A$-isotopic to $h_{r,(1,\ep_2)}$.
We show $h_{r,(1,\ep_2)}$ is $\A$-isotopic to $h_{0,r}$.
Set $\ep_2=-1$.
Considering $\pi$-rotations of
$$
\begin{array}{l}
\big\{
\underbrace{\underbrace{1,3\red{,}\ldots,r}_{\substack{\text{odd}}},\ 
\underbrace{{r+1},{r+3},\ldots,{2r}}_{\substack{\text{odd}}},\ 
\ldots,\ 
\underbrace{{(a-1)r+1},{(a-1)r+3},\ldots,{ar}}_{\substack{\text{odd}}}
}_{\substack{\text{odd}}},\\
\hspace{70mm}
\underbrace{{ar+1},{ar+3},\ldots,{ar+r-2}}_{\substack{\text{even}}},\ 
{m}\big\}
\end{array}
$$
on the source, we see that $h_{r,(1,\ep_2)}$ is $\A$-isotopic to
$$\big(\ep_2x_1,x_2,\ep_2x_3,x_4,\ldots,\ep_2x_{m-2},x_{m-1},
h_1(x),\ldots,
h_{n-m}(x),\ep_2h_{n-m+1}(x)\big),
$$
noticing $(a+1)r=m$.
By $\pi$-rotations on the target,
we have the result.
\end{proof}
\begin{proof}[Proof of the claim of the second part of Case\/ $3$]
Let us assume $r=4l+2$ and $a$ is even.
By Proposition \ref{prop:iso1},
$h_{r,(\ep_1,\ep_2)}$ is $\A$-isotopic to $h_{r,(\ep_1,1)}$.
We show $h_{r,(\ep_1,1)}$ is $\A$-isotopic to $h_{0,r}$.
Set $\ep_1=-1$.
Considering $\pi$-rotations of
$$
\begin{array}{l}
\big\{1,
\underbrace{2,4\red{,}\ldots,r}_{\substack{\text{odd}}},\\
\hspace{5mm}
\underbrace{
\underbrace{{r+1},{r+3},\ldots,{2r-1}}_{\substack{\text{odd}}},\ 
\ldots,\ 
\underbrace{{r(a-1)+1},{r(a-1)+3},\ldots,{r(a-1)+r-1}
}_{\substack{\text{odd}}}}_{\substack{\text{odd}}},\\
\hspace{75mm}
\underbrace{{ar+2},{ar+4},\ldots,{ar+r-2}}_{\substack{\text{even}}},{m}
\big\}
\end{array}
$$
we see that
$h_{r,(\ep_1,1)}$ is $\A$-isotopic to
$$
\big(x_1,\ep_1x_2,x_3,\ep_1x_4\red{,}\ldots,\ep_1x_{m-2},x_{m-1},
\ep_1h_1(x)
,h_2(x)\red{,}\ldots,
h_{n-m}(x),\ep_1h_{n-m+1}(x)\big).
$$
By $\pi$-rotations on the target,
we have the result.
\end{proof}
\begin{proof}[Proof of the claim of Case\/ $4$]
Let us assume $r=4l+3$.
By Proposition \ref{prop:iso1},
$h_{r,(\ep_1,\ep_2)}$ is $\A$-isotopic to $h_{r,(1,\ep_2)}$.
We show $h_{r,(1,\ep_2)}$ is $\A$-isotopic to $h_{0,r}$.
Set $\ep_2=-1$.
Considering $\pi$-rotations of
$$
\begin{array}{l}
\big\{
\underbrace{1,3,\ldots,r}_{\substack{\text{even}}},\ 
\ldots,\ 
\underbrace{{a(r-1)+1},{a(r-1)+3},\ldots,{a(r-1)+r}
}_{\substack{\text{even}}},\\
\hspace{60mm}
\underbrace{{ar+1},{ar+3},\ldots,{ar+r-2}}
_{\substack{\text{odd}}},\ 
m\big\},
\end{array}$$
we see that
$h_{r,(1,\ep_2)}$ is $\A$-isotopic to 
$$
(\ep_2x_1,x_2,\ep_2x_3,\ldots,\ep_2x_{m-2},x_{m-1},
h_1(x),\ldots,\ep_2h_{n-m+1}(x)\big).
$$
By $\pi$-rotations on the target,
we have the result.
\end{proof}
\medskip

{\bf Acknowledgement}
The author thanks Toru Ohmoto for fruitful advices.
He also thanks Yasutaka Nakanishi for constant
encouragement.
The author thanks the referee for a very careful reading.


\medskip

{\small 
\begin{flushright}
\begin{tabular}{l}
Department of Mathematics,\\
Graduate School of Science, \\
Kobe University, \\
Rokko, Nada, Kobe 657-8501, Japan\\
  E-mail: {\tt sajiO\!\!\!amath.kobe-u.ac.jp}
\end{tabular}
\end{flushright}
}

\end{document}